\documentclass[12pt]{article}
\usepackage{mathtools,fullpage,amsthm,amssymb,enumerate}
\usepackage[usenames,dvipsnames]{color}

\expandafter\ifx\csname dplus\endcsname\relax \csname newbox\endcsname\dplus\fi
\expandafter\ifx\csname dplustemp\endcsname\relax
\csname newdimen\endcsname\dplustemp\fi
\setbox\dplus=\vtop{\vskip -8pt\hbox{%
    \kern .2em
    \special{pn 6}%
    \special{pa 0 50}%
    \special{pa 50 100}%
    \special{fp}%
    \special{pa 50 100}%
    \special{pa 100 50}%
    \special{fp}%
    \special{pa 100 50}%
    \special{pa 50 0}%
    \special{fp}%
    \special{pa 50 100}% 
    \special{pa 50 0}%
    \special{fp}%
    \special{pa 50 0}% 
    \special{pa 0 50}%
    \special{fp}%
    \special{pa 0 50}%
    \special{pa 100 50}%
    \special{fp}%
    \hbox{\vrule depth0.080in width0pt height 0pt}%
    \kern .7em
  }%
}%
\def\gjoin{\copy\dplus}

\newcommand{\caze}[2]{\textbf{Case {#1}:} \textit{#2}}
\usepackage{xcolor}
\usepackage{dsfont}

\usepackage{marvosym}

\newtheorem{proposition}{Proposition}[section]
\newtheorem{corollary}[proposition]{Corollary}
\newtheorem{conjecture}[proposition]{Conjecture}
\newtheorem{lemma}[proposition]{Lemma}
\newtheorem{theorem}[proposition]{Theorem}
\newtheorem{observation}[proposition]{Observation}
\theoremstyle{definition}
\newtheorem{example}[proposition]{Example}
\newtheorem{definition}[proposition]{Definition}
\newtheorem{remark}[proposition]{Remark}

\newcommand{\sizeof}[1]{\left\lvert{#1}\right\rvert}
\newcommand{\floor}[1]{\left\lfloor{#1}\right\rfloor}
\newcommand{\ceil}[1]{\left\lceil{#1}\right\rceil}
\newcommand{\chisc}{\chi_{\mathrm{SC}}}
\newcommand{\chisp}{\chi_{\mathrm{SP}}}

\newcommand{\csg}[2]{\overline{K}_{#1} \gjoin K_{#2}}

\def\C#1{\left|{#1}\right|}
\def\chil{\chi_\ell}

\def\spo{\operatorname{\mathaccent"017{s}}}
\def\st{\colon\,}

\def\esub{\subseteq}
\def\nul{\varnothing}
\def\FR{\frac}
\def\FL{\floor}

\def\cost{sum-color cost}

\def\NN{{\mathbb N}}

\def\GG{{\mathbb G}}
\def\Kb{\overline{K}}
\def\Gb{\overline{G}}
\def\Hb{\overline{H}}
\def\CH{\binom}

\def\SE#1#2#3{\sum_{#1=#2}^{#3}}
\def\cart{{\scriptstyle\Box}}
\def\ntalph{\FR{n^2}{2\alpha(G)}+\FR n2}
\def\cS{{\mathcal S}}
\def\hj{{J}}

\author{Thomas Mahoney\thanks{Emporia State University, Emporia, KS,
tmahoney@emporia.edu}\,,
Gregory J.~Puleo\thanks{Auburn University, Auburn, AL, gjp0007@auburn.edu.
Previously at Coordinated Science Laboratory,
University of Illinois, Urbana, IL.
Research supported by the IC Postdoctoral Fellowship.}\,,
Douglas B.~West\thanks{Zhejiang Normal University, Jinhua, China, and
University of Illinois, Urbana, IL, west@math.uiuc.edu.
Research supported by Recruitment Program of Foreign Experts,
1000 Talent Plan, State Administration of Foreign Experts Affairs, China.
}}
\title{Online Paintability: The Slow-Coloring Game}
\begin{document}
\maketitle

\vspace{-2pc}
\begin{abstract}
The {\it slow-coloring game} is played by Lister and Painter on a graph $G$.
On each round, Lister marks a nonempty subset $M$ of the uncolored vertices,
scoring $\C M$ points.  Painter then gives a color to a subset of $M$ that is
independent in $G$.  The game ends when all vertices are colored.  Painter and
Lister want to minimize and maximize the total score, respectively.  The best
score that each player can guarantee is the \emph{\cost} of $G$, written
$\spo(G)$.  The game is an online variant of online sum list coloring.

We prove $\FR{\C{V(G)}}{2\alpha(G)}+\FR 12\le\FR{\spo(G)}{\C{V(G)}}\leq 
\max\left\{\frac{\C{V(H)}}{\alpha(H)}\st\!H \subset G\right\}$, where
$\alpha(G)$ is the independence number, and we study when equality holds in
the bounds.  We compute $\spo(G)$ for graphs with
$\alpha(G)=2$.  Among $n$-vertex trees, we prove that $\spo$ is minimized by
the star and maximized by the path.  We also obtain good bounds on
$\spo(K_{r,s})$.
%Trivially $\spo(G)$ is at
%most the sum-paintability of $G$, with equality only when all components of $G$
%are complete.  Also $\spo(G)$ is at least the chromatic sum (the minimum sum of
%vertex colors in a proper coloring by positive integers), with equality if
%$\alpha(G)\le2$.
\end{abstract}

%%Talk abstract
%Online Sum Paintability: The Slow-Coloring Game
%Douglas B. West
%Zhejiang Normal University and University of Illinois
%
%The "slow-coloring game" is played by Lister and Painter on a graph $G$.
%On each round, Lister marks a nonempty subset $M$ of the remaining vertices,
%scoring $|M|$ points.  Painter then deletes a subset of $M$ that is independent
%in $G$.  The game ends when all vertices are deleted.  Painter's goal is to 
%minimize the total score; Lister seeks to maximize it.  The score that each 
%player can guarantee doing no worse than is the "slow-color cost" of $G$,
%written $s(G)$.  The game is a variant of online list coloring.
%
%We have obtained the following results.  Trivially lower and upper bounds on
%$s(G)$ are the chromatic sum and the sum-paintability, with equality in
%the lower bound when $\alpha(G)\le2$ and equality in the upper bound if and
%only if all components are complete (the "chromatic sum" is the minimum sum of 
%vertex colors in a proper coloring by positive integers).  We give sharp upper 
%and lower bounds on $s(G)$ in terms of the independence number.  Among 
%$n$-vertex trees, $s(G)$ is minimized by the star and maximized by the path 
%(where it equals $3n/2$).  We give good bounds on $s(K_{r,s})$.  These results 
%are joint with Thomas Mahoney and Gregory Puleo.  We will also discuss a later 
%linear-time algorithm to compute $s(G)$ exactly when $G$ is a tree, which
%is joint work with Gregory Puleo.

\baselineskip 16pt
\section{Introduction}\label{sec:intro}
A {\it proper coloring} of a graph $G$ assigns each vertex in the vertex
set $V(G)$ a color so that adjacent vertices have distinct colors.  That is,
the set of vertices assigned a given color must be an {\it independent set},
meaning a set of pairwise nonadjacent vertices.  The chromatic number, written
$\chi(G)$, is the least $k$ such that $G$ has a proper coloring using $k$
colors.

To examine worst-case behavior of proper coloring when not all colors are
available at all vertices, we study a coloring game played by Lister and
Painter on a graph $G$.  In the $i$th round, Lister marks a nonempty subset $M$
of the uncolored vertices as eligible to receive color $i$, scoring $\C M$.
Painter then gives color $i$ to a subset of $M$ that is independent in $G$.
The game ends when all vertices are colored, producing a proper coloring.
Painter's goal is to minimize the total score; Lister seeks to maximize it,
thereby introducing a large total delay in the coloring process.  We call this
the {\it slow-coloring game}.  The score that each player can guarantee doing
no worse than is the \emph{\cost} of $G$, written $\spo(G)$.

The slow-coloring game arose as an online variant of online sum list coloring.
List coloring generalizes the classical model of graph coloring by introducing
a \textit{list assignment} $L$ that assigns to each vertex $v$ a set $L(v)$ of
available colors.  A graph $G$ is \emph{$L$-colorable} if it has a proper
coloring $\phi$ such that $\phi(v) \in L(v)$ for every vertex $v$.  Given
$f\st V(G)\to\NN$, a graph $G$ is \emph{$f$-choosable} if $G$ is $L$-colorable
whenever $\C{L(v)}\ge f(v)$ for every vertex $v$.

Introduced by Vizing~\cite{V2} and by Erd\H{o}s, Rubin, and Taylor~\cite{ERT},
the \textit{choice number} or \textit{choosability} $\chil(G)$ is the least
$k$ such that $G$ is $f$-choosable whenever $f(v)\ge k$ for all $v\in V(G)$.
Since one option for the lists is to be identical, always $\chil(G)\ge\chi(G)$.
There are already hundreds of papers on aspects of choosability.

Instead of minimizing the threshold list size, we may seek the least sum (or
average) of list sizes.  Introduced by Isaak~\cite{I1,I2} and studied
in~\cite{BBBD,H1,H2}, the \textit{sum-choosability} of a graph $G$, denoted
$\chisc(G)$, is the minimum of $\sum f(v)$ over all $f$ such that $G$ is
$f$-choosable.

Using integer colors, a list assignment can be viewed as a schedule, presenting
$\{v\st i\in L(v)\}$ on round $i$ as the set $M_i$ of candidates to receive
color $i$.  In slow-coloring, Lister can model any list assignment, but Lister
has other options and can change strategy in response to Painter's moves.
Such flexibility for Lister leads to the {\it $f$-painting game}, introduced
independently by Schauz~\cite{S1} and by Zhu~\cite{Z1}.  As as in the
slow-coloring game, Lister marks a set $M$ and Painter colors an independent
subset of $M$.  Instead of fixed lists, we have a size $f(v)$ for each vertex
$v$ as the number of times Painter can allow $v$ to be marked.  Lister wins the
game by marking some vertex $v$ more than $f(v)$ times; Painter wins by
coloring all the vertices before that happens.  The graph is
\textit{$f$-paintable} if Painter has a winning strategy.

The \textit{paint number} or \textit{paintability} is the least $k$ such that
$G$ is $f$-paintable whenever $f(v)\ge k$ for all $v\in V(G)$.  The \emph{sum
paintability} of a graph $G$, introduced by Carraher, Mahoney, Puleo, and
West~\cite{CMPW} and written $\chisp(G)$, is the minimum of $\sum f(v)$ over
all $f$ such that $G$ is $f$-paintable.  Finding $f$ so that $G$ is
$f$-paintable corresponds to giving supplies to each vertex; $f(v)$ tokens are
allocated to $v$, and one is used each time $v$ is marked.  In
sum-paintability, Painter seeks to minimize the total number of tokens
allocated.

The slow-coloring game differs from sum-paintability in that Painter need not
allocate tokens in advance.  Painter allocates tokens to vertices in response
to the marked set.  Since Painter can choose independent sets as dictated by an
optimal strategy for sum-paintability, always $\spo(G) \leq \chisp(G)$.  Note
also that $\chisc(G)\le\chisp(G)$, since Painter can win the $f$-painting game
in which Lister's moves are specified by $L$ if and only if $G$ is
$L$-colorable (lists are truncated when vertices are colored).  The inequality
$\spo(G) \geq \chisc(G)$ seems natural, but we will see in
Example~\ref{example} that it is not always true.

Unlike sum-paintability, \cost\ is given by
an easily described (but hard to compute) recursive formula.  The key point is
that prior choices do not affect Painter's optimal strategy for coloring
subsets of marked sets on the uncolored subgraph.  Thus we can view colored
vertices as having been ``deleted'' from the graph.
\begin{proposition}\label{pr:recur}
%\begin{equation} \label{eq:recursive}
$\qquad\spo(G)=\displaystyle{\max_{\nul\ne M\esub V(G)}}
\left(\C{M} + \min\,\spo(G-I)\right).$
%\end{equation}
%where the minimum is over subsets $I\esub M$ such that $I$ is independent in
%$G$.
\end{proposition}
\begin{proof}
In response to the initial marked set $M$, Painter minimizes the additional
score over colored subsets $I\esub M$ such that $I$ is independent in $G$.
Lister chooses $M$ to maximize the resulting total score.
\end{proof}

In studying optimal strategies for Lister and Painter, simple observations
reduce the set of moves that need to be considered.

\begin{observation}\label{simple}
On any graph, there are optimal strategies for Lister and Painter such that
Lister always marks a set $M$ inducing a connected subgraph, and Painter
always colors a maximal independent subset of $M$.
\end{observation}
\begin{proof}
A move in which Lister marks a disconnected set $M$ can be replaced with
successive moves marking the vertex sets of the components of the subgraph
induced by $M$.  Also, coloring extra vertices at no extra cost cannot hurt
Painter.
\end{proof}

Another easy observation sometimes yields a useful lower bound.

\begin{observation}\label{disjoint}
If $G_1$ and $G_2$ are disjoint subgraphs of $G$, then 
$\spo(G)\ge\spo(G_1)+\spo(G_2)$.
\end{observation}
\begin{proof}
Lister can play an optimal strategy on $G_1$ while ignoring the rest and then
do the same on $G_2$, achieving the score $\spo(G_1)+\spo(G_2)$.
\end{proof}

\begin{observation}\label{monoton}
If $G$ is a subgraph of $H$, then $\spo(G)\le \spo(H)$.
\end{observation}
\begin{proof}
On $G$, Painter can play an optimal strategy for the supergraph $H$.
\end{proof}

The average cost per vertex is $\FR{\spo(G)}{\C{V(G)}}$.
In Section~\ref{sec:general}, we prove easy general bounds.
Let $\alpha(G)$ denote the maximum size of an independent set in a graph $G$.
%The {\it complement} of a graph $G$ is the graph $\Gb$ with $V(\Gb)=V(G)$
%such that vertices are adjacent in $\Gb$ if and only if they are not adjacent
%in $G$.  the {\it degree} $d_G(v)$ of a vertex $v$ in a graph $G$ is the
%number of incident edges (our graphs have no loops or multiedges).
%A graph is {\it regular} if all vertices have the same degree.

\begin{theorem}\label{sharpbds}
The following are sharp bounds on $\spo(G)$:
\[
\FR{\C{V(G)}}{2\alpha(G)}+\FR 12\le\FR{\spo(G)}{\C{V(G)}}\leq 
\max\left\{\frac{\C{V(H)}}{\alpha(H)} \st H \esub G\right\}.
\]
Equality holds in the upper bound if and only if $G$ has no edges.
Among complete multipartite graphs that are regular, equality holds in the
lower bound if and only if $\chi(G)=1$ or $\alpha(G)\le 2$.
\end{theorem}

Although the upper bound holds with equality only for edgeless graphs, it may
be asymptotically sharp for the complete bipartite graph $K_{r,r}$ and other
graphs; see Section~\ref{sec:bip}.

Let $\rho(G)= \max\left\{\frac{\C{V(H)}}{\alpha(H)}\st H\esub G\right\}$.
Trivially, $\rho(G)\le \chi(G)$, so $\FR{\spo(G)}{\C{V(G)}}\le\chi(G)$.
The quantity $\rho(G)$ has been called the \emph{Hall ratio} of $G$,
defined in \cite{HR2} and explored further in \cite{HR1,HR3,HR4,HR5} (in fact,
also $\rho(G)\le\chi^*(G)$, where $\chi^*(G)$ is the fractional chromatic
number).  By well-known results on $\alpha(G)$ and $\chi(G)$ for random graphs
(see Section~\ref{sec:general}), Theorem~\ref{sharpbds} implies that with high
probability $\FR{\spo(G)}{\C{V(G)}}$ is within a constant multiple of $\chi(G)$.

In Section~\ref{sec:alpha2}, we compute $\spo(G)$ for all $G$ with
$\alpha(G)=2$, and this aids in studying sharpness of the lower bound in
Theorem~\ref{sharpbds}.  A {\it matching} is a set of pairwise disjoint edges.

\begin{theorem}\label{thm:alpha}
If $\alpha(G)=2$, then $\spo(G)={n-q+1\choose 2}+{q+1\choose 2}$, where $q$ is
the maximum size of a matching in the complement of $G$.
\end{theorem}

We have noted that Painter can follow a winning strategy in an $f$-painting
game to achieve $\spo(G)\le\chisp(G)$.  In Section~\ref{sec:strict}, we
characterize equality.

\begin{theorem}\label{chisp}
$\spo(G)=\chisp(G)$ if and only if every component of $G$ is complete.
\end{theorem}

The two upper bounds $\spo(G)\le\C{V(G)}\rho(G)$ and $\spo(G)\le\chisp(G)$ are
independent.

\begin{example}\label{example}
Note that $\spo(K_r)=\CH{r+1}2$; Painter can only color one vertex in each
round, so on this graph it is optimal for Lister to always mark all uncolored
vertices.

When $G$ is the disjoint union of the complete graph $K_r$ and an independent
set of $n-r$ vertices, we have $\rho(G)=r$, but $\spo(G)=\chisp(G)= n+\CH r2$,
so $\C{V(G)}\rho(G)$ is much bigger.

When $G$ is bipartite with at least one edge, $\rho(G)=2$.  When
$G=K_{2,r}$, we have $\C{V(G)}\rho(G)=2r+4$, but
$\chisp(G)\approx 2r+2\sqrt r$ (see~\cite{CMPW}), so $\chisp(G)$ is
bigger.  More generally, F\"uredi and Kantor~\cite{FK} proved for
$a\ge3$ and $r>50a^2\log a$ that
$\chisc(K_{a,r})\ge 2r+.068a\sqrt{r\log a}$, so $\chisp(K_{a,r})$ is
even larger, while $\C{V(K_{a,r})}\rho(K_{a,r})$ is only $2r+2a$.
Hence in particular also $\chisc(K_{a,r})>\spo(K_{a,r})$ when
$r> 50a^2\log a$ and $a$ is sufficiently large, by Theorem~\ref{sharpbds}.

On these complete bipartite examples, $\chisp$ is larger than
$\C{V(G)}\rho(G)$, but the two bounds are asymptotically equal.  Another
bipartite example gives asymptotic ratio $5/4$.  Let $G=P_k\cart K_2$, where
$\cart$ denotes the cartesian product ($G$ is the $2$-by-$k$ ``grid'').  Note
that $G$ can be constructed from $K_2$ by successively adding ears of length
$3$.  A lemma in \cite{CMPW} shows that each such addition increases the
sum-paintability by $5$, so $\chisp(G)=5k-2$.  However, $\C{V(G)}\rho(G)=4k$.
\end{example}

In Section~\ref{sec:treeupper}, we prove sharp bounds
on the \cost\ of $n$-vertex trees.

\begin{theorem}\label{treethm}
Among $n$-vertex trees, the value of $\spo$ is minimized
by the star and maximized by the path.  Furthermore, with
$u_t=\floor{(-1+\sqrt{8t+1})/2}$ and $T$ being an $n$-vertex tree,
\[
n+\sqrt{2n}~\approx~ n+u_{n-1}~ = ~ \spo(K_{1,n-1})
~\le~\spo(T)~\leq~\spo(P_n) ~=~ \floor{{3n}/2}.
\]
\end{theorem}

In a subsequent paper, Puleo and West~\cite{PW} provide a linear-time algorithm
to compute $\spo$ on trees, via an inductive formula.  The formula yields 
characterizations of the minimizing $n$-vertex trees and the trees attaining
the maximum value $3n/2$ (which requires $n$ even).  The star is the unique
$n$-vertex tree minimizing $\spo$ whenever $n-1$ and $n-2$ are not of the form
$\CH k2$.  The $n$-vertex trees $T$ with $\spo(T)=3n/2$ are those having a
spanning acyclic subgraph in which every vertex has degree $1$ or $3$.

We do not know the complexity of computing $\spo(G)$ in general or on larger
families than trees.  It is not obvious that the decision problem is in NP.

We conjecture that Theorem~\ref{treethm} generalizes to $k$-trees.
A \textit{$k$-tree} is a graph obtained from $K_k$ by iteratively adding a
vertex whose neighborhood is a $k$-clique in the existing graph.
The \textit{join} $G\gjoin H$ of graphs $G$ and $H$ is obtained from the
disjoint union $G+H$ by making each vertex in $G$ adjacent to each vertex in
$H$.  The \textit{$r$th power} of $G$ is the graph $G^r$ with vertex set $V(G)$
where vertices are adjacent if and only the distance between them in $G$ is at
most $r$.  The graphs $K_k\gjoin \Kb_{n-k}$ and $P_n^k$ are $k$-tree analogues
of $n$-vertex stars and paths.  Our argument to compute $\spo(K_{1,n-1})$
allows us more generally to compute $\spo(K_k\gjoin \Kb_{n-k})$. 
%While trees lend themselves easily to induction, it seems considerably
%harder to use induction to prove results about $\spo$ on $k$-trees.
\begin{theorem}\label{thm:split1}
Let $u_t=\floor{\FR{-1+\sqrt{8t+1}}2}$.  For $r,s\in\NN$,
\[ \spo(\csg{r}{s})=r+{s+1\choose2}+s u_r. \]
\end{theorem}

\begin{conjecture}
For $k\in\NN$ and any $k$-tree $T$ with $n$ vertices,
  \[ \spo(K_k\gjoin \Kb_{n-k})~\leq~\spo(T)~\leq~\spo(P^{k}_n). \]
\end{conjecture}

An easy lower bound for $\spo(P_n^k)$ follows from Observation~\ref{disjoint}.
Since $P_n^k$ contains $\FL{\FR n{k+1}}$ disjoint copies of $K_{k+1}$, we have
$\spo(P_n^k)\ge\spo(\floor{\FR n{k+1}}K_{k+1})+\spo(K_r)$, where
$r\equiv n\mod{(k+1)}$, and we conjecture that equality holds.  This formula
reduces to the correct answer for $k=1$.

In Section~\ref{sec:bip}, we study the slow-coloring game on complete bipartite
graphs.
\begin{theorem}\label{thm:bip}
Let $u_t=\floor{\FR{-1+\sqrt{8t+1}}2}$.  For $r\ge s>0$,
\[
r+\FR{5s-3}2+u_{r-s}\le\spo(K_{r,s})\le r+s+2\sqrt{rs}.
\]
%\[\spo(K_{r,s}) \geq r+s+\max_{\substack{r_1+\cdots+r_s=r\\r_i \in \nats}}\sum u_{r_i}.\]
%For $r=s>0$, the lower bound is $3.5r-1$.
\end{theorem}

Supported by computational data, we conjecture that the upper bound is
asymptotically optimal, at least when $r=s$, yielding $\spo(K_{r,r})\sim 4r$.
This bound is asymptotic to the upper bound in Theorem~\ref{sharpbds} (and to
the number of vertices times the chromatic number).  We ask whether more
generally the complete $k$-partite subgraphs with parts of equal size satisfy
$\spo(G)\sim k\C{V(G)}$ as $\C{V(G)}\to\infty$.

In a further subsequent paper, Gutowski et al.~\cite{GKWZZ} study the bounds on
$\spo(G)$ for families of graphs containing forests.  A graph is
{\it $d$-degenerate} if every subgraph has a vertex of degree at most $d$.
Inductively, $\chi(G)\le 1+d$ when $G$ is $d$-degenerate, so 
Theorem~\ref{sharpbds} yields $\spo(G)\le (1+d)\C{V(G)}$ when $G$ is 
$d$-degenerate, and the disjoint union of copies of $K_{d+1}$ show that the
value can be as high as $(1+d/2)\C{V(G)}$.  In~\cite{GKWZZ}, the bound
$\spo(G)\le (1+3d/4)\C{V(G)}$ is proved for $d$-degenerate graphs.

A graph is {\it outerplanar} if it embeds in the plane with all vertices
lying on a single face.  It is an elementary exercise that outerplanar graphs
are $2$-degenerate, so the result mentioned above implies $\spo(G)\le 2.5n$
when $G$ is an $n$-vertex outerplanar graph.  In~\cite{GKWZZ}, the bound is
improved to $7n/3$.  Disjoint copies of $K_3$ show that the value can be as
large as $2n$, which is conjectured optimal.

The famous Four Color Theorem states $\chi(G)\le 4$ when $G$ embeds in the
plane, so $\spo(G)\le 4n$ for $n$-vertex planar graphs.  In~\cite{GKWZZ}, the
bound is improved to $3.9857n$, where the coefficient more precisely is
$(13+4\sqrt3)/5$.  Disjoint copies of $K_4$ show that the value can be as large
as $2.5n$, which is conjectured optimal.

\section{General Bounds}\label{sec:general} \label{sec:alpha2}
We begin with a stronger lower bound than stated in Section~\ref{sec:intro},
using the \emph{chromatic sum} of $G$, defined by Kubicka~\cite{KubT}
(see~\cite{Kub} for a survey).
\begin{definition}
The \emph{chromatic sum} of a graph $G$, written $\Sigma(G)$, is the minimum of
$\sum_{v \in V(G)}c(v)$ over all proper colorings $c$ using positive integers.
\end{definition}

The chromatic sum is the outcome of the slow-coloring game when Lister
follows the strategy of always marking the entire remaining graph.
Let $G[M]$ denote the subgraph of $G$ induced by a vertex subset $M$.

\begin{theorem}\label{thm:general}
For every $n$-vertex graph $G$,
\[ \Sigma(G) \leq \spo(G) \leq n\rho(G). \]
Equality holds in the lower bound when all components of $G$ are complete.
Equality holds in the upper bound when $G$ has no edges.
\end{theorem}
\begin{proof}
Given that Lister always marks all remaining vertices, let $V_i$ be the set
of vertices colored by Painter on round $i$.  The vertices in $V_i$ are marked
$i$ times.  Thus the total cost is at least $\Sigma(G)$.  Equality holds for
a disjoint union of complete graphs, because Painter always colors one vertex 
in each component having a marked vertex.

For the upper bound, let $r = \rho(G)$.  Given any marked set $M$, the
\emph{greedy strategy} for Painter colors a largest independent set in $G[M]$.
The definition of $\rho(G)$ yields $\alpha(G[M])\ge \C{M}/r$.  For any game
played against this strategy, let $m_1,\ldots,m_t$ be the sizes of the marked
sets in the successive rounds.  In round $i$ Painter colors at least $m_i/r$
vertices, so $\sum_{i=1}^t \FR{m_i}r \leq n$.  Multiplying by $n$ shows that
Lister scores at most $nr$.

Equality for $G=\Kb_n$ is trivial.  Conversely, if equality holds in the upper
bound, then exactly $m_i/r$ vertices must be colored in round $i$.  In
particular, in the last round, all $m_i$ marked vertices are colored, requiring
$m_i/r=m_i$, and hence $r=1$.  For any graph $G$ having an edge, $\rho(G)\ge2$,
so equality holds only for edgeless graphs.
\end{proof}

\begin{corollary}\label{n2alph}
%When $G$ is an $n$-vertex graph, $\Sigma(G)\ge \FR n2(1+\FR n{\alpha(G)})$,
%and hence \[ \spo(G)\ge\frac{n^2}{2\alpha(G)} + \frac{n}{2}. \]
$\Sigma(G)\ge \FR n2(1+\FR n{\alpha(G)})$ when $G$ has $n$ vertices, and 
hence $\spo(G)\ge\frac{n^2}{2\alpha(G)} + \frac{n}{2}$.
\end{corollary}
\begin{proof}
We use induction on $n$; the claim is trivial for $n=0$.  For $n>0$, let $I$ be
the set of vertices receiving color $1$ in a proper coloring of $G$ with
minimum sum, and let $a=\alpha(G)$.  Note that
$\Sigma(G)=\C I+(n-\C I)+\Sigma(G-I)$.  Using the induction hypothesis, 
$\Sigma(G)\ge n+\FR12(n-\C I)(1+\FR{n-\C I}{\alpha(G-I)})$.  Minimizing the 
numerator and maximizing the denominator, we have
$\Sigma(G)\ge n+\FR12(n-a)(1+\FR{n-a}a)=\FR n2(1+\FR na)$.
\end{proof}
%Let $a = \alpha(G)$ and $q = \floor{n/a}$.  Let $c$ be a proper coloring that
%minimizes $\sum_{v \in V(G)}c(v)$.  Let $V_i=\{v\in V(G)\st c(v)=i\}$, so
%$\sum_{v\in V(G)}c(v)=\sum i\C{V_i}=\sum \C{\bigcup_{j\ge i} V_j}$.  Since
%$\sizeof{V_i} \leq a$ for all $i$, we have $\C{\bigcup_{j\ge i} V_j}\ge
%n-ja$.  Let $\epsilon = \FR na -q$; note that $0\le \epsilon<1$.  We compute
%   \begin{align*}
%     \Sigma(G) &\geq (q+1)n - \sum_{i=1}^q ia =(q+1)\left(n-\FR{qa}2\right)
%=(q+1)\left(n-\FR{a\epsilon-n}2\right)\\
%     &=\left(\FR na -\epsilon+1\right)\FR{n+a\epsilon}2=
%     \frac{n^2}{2a} + \frac{n}{2} + \FR{a\epsilon}2(1-\epsilon) 
%     \geq \frac{n^2}{2a} + \frac{n}{2\hangs.}
%   \end{align*}
%\vspace{-3pc}
%\end{proof}
%\bigskip
%***DBW - The value of this corollary is not worth the space that would be
%required to explain it.
%\begin{corollary}
%  For any graph $G$, \[\spo(G) \leq n\chif(G) \leq n\chi(G),\] where
% $\chif(G)$ is the fractional chromatic number of $G$.
%\end{corollary}

The \textit{binomial random graph model} (see~\cite{bollobas-book}) is the
probability space $\GG(n,p)$ generating graphs with vertex set $\{1,\dots,n\}$
by letting vertex pairs
be edges with probability $p$, independently.  An event occurs \textit{with
high probability} if its probability in $\GG(n,p)$ tends to $1$ as $n\to\infty$.

\begin{corollary}
For fixed $p \in (0,1)$, there is a positive constant $c$ such that for
$G$ sampled from $\GG(n,p)$, with high probability
$c\chi(G) \leq \FR{\spo(G)}n \leq \chi(G)$.
\end{corollary}
\begin{proof}
The upper bound always holds, by Theorem~\ref{thm:general}, since
$\chi(G)\ge \C{V(G)}/\alpha(G)$.  For the lower bound, it suffices by
Corollary~\ref{n2alph} to obtain a constant $c'$ such that
$c'\chi(G)\leq \frac{n}{2\alpha(G)} + \frac{1}{2}$ with high probability.
By well-known results on the concentration of the clique number and the
chromatic number in
$\GG(n,p)$~\cite{bollobas-chromatic,bollobas-clique,bollobas-book}, there are
positive constants $c_1$ and $c_2$ (depending on $p$) such that for any
positive $\epsilon$, with high probability $\chi(G)$ is within a fraction
$1+\epsilon$ of $c_1 \frac{n}{\log n}$ and $\alpha(G)$ is within a
fraction $1+\epsilon$ of $c_2 \log n$.  The result follows.
\end{proof}

%The ratio between the bounds in Theorem~\ref{thm:general} can be arbitrarily
%large.  Consider $G$ obtained from $K_p$ by adding $q$ isolated vertices.
%Note that $\rho(G)=\chi(G)=p$ but $\Sigma(G)=\Sigma(K_p)+q=(p^2+p)/2+q$.
%As $q$ tends to $\infty$ for fixed $p$, we have
%$|V(G)|\rho(G)/\Sigma(G)\sim p$, and $p$ can be arbitrarily large.

% Toward understanding the quality of bounds, it helps to determine when
%equality holds.
%DBW - although this is good motivation, we don't do it,
%and hence we should not say it.

We next determine $\spo(G)$ when $\alpha(G)\le2$, showing that the lower bound
from Theorem~\ref{thm:general} holds with equality in that case.
Let $\Gb$ denote the complement of a graph $G$.

\begin{lemma}\label{alpha2Sig}
If $\alpha(G) \leq 2$, then $\Sigma(G) = {n-q+1 \choose 2} + {q+1 \choose 2}$,
where $q$ is the maximum size of a matching in $\Gb$.
\end{lemma}
\begin{proof}
When $\alpha(G)\le2$, all independent sets have size at most $2$.  To minimize
the sum of the colors, the color classes of size $2$ should be given the lowest
colors, and the largest possible number of disjoint classes of size $2$ should
be used.  This largest number is $q$, and the remaining vertices must have
distinct colors.  Thus $q$ vertices receive colors $1$ through $q$, and the
remaining vertices receive colors $1$ through $n-q$.  The result follows.
\end{proof}

%Here we compute $\spo(G)$ when $\alpha(G)\le 2$.  As Theorem~\ref{thm:alpha2}
%said, $\spo(G) = \Sigma(G) = {n-q+1 \choose 2} + {q+1 \choose 2}$, where $q$
%is the maximum size of a matching in $\Gb$.
%The proof consists of two lemmas.  We first show that if $\alpha(G) \leq 2$,
%then $\Sigma(G) = {n-q+1 \choose 2} + {q+1 \choose 2}$.  We then show that if
%$G$ is a complete multipartite graph with $\alpha(G)=2$, then
%$\spo(G) \leq {n-q+1 \choose 2} + {q+1 \choose 2}$.  Since $\spo(H)\le\spo(G)$
%when $H\esub G$, and always $\spo(G)\ge \Sigma(G)$, this proves
%Theorem~\ref{thm:alpha2}.

\begin{lemma}\label{lem:multipart2}
If $G_{n,q}$ is the $n$-vertex complete multipartite graph with $q$ parts of
size $2$ and $n-2q$ parts of size $1$, then
$\spo(G_{n,q}) = {n-q+1 \choose 2} + {q+1 \choose 2}$.
\end{lemma}
\begin{proof}
The lower bound follows from Theorem~\ref{sharpbds} and Lemma~\ref{alpha2Sig}.
For the upper bound, let $f(n,q) = \CH{n-q+1}2+\CH{q+1}2$.  We prove
$\spo(G_{n,q}) \leq f(n,q)$ by induction on $n$.  The claim is trivial
for $n\le1$.

%To prove an upper bound on $\spo(G_{n,q}$, we give a strategy for Painter.
%It suffices to give an good response to the first move, since the remainder
%of the game is a game on a graph of the same type.
%When $n=1$ the claim is trivial, so consider $n>1$.

%When $r=2$, we prove $\spo(G)=2\CH{t+1}2$.  More generally, let $H_{t,p}$ be
%the complete $t$-partite graph with $p$ parts of size $2$ and $t-p$ parts of
%size $1$; the desired graph $G$ is $H_{t,t}$.  We prove $\spo(H_{t,p})=f(t,p)$,
%where $f(t,p)=\CH{t+1}2+\CH{p+1}2$.  Lister ensures the lower bound by playing
%separately on a $t$-clique and a $p$-clique covering $V(H_{t,p})$.

For $n>1$, let $s$ be the total score of the game.  Consider the first round.
If Lister marks some two nonadjacent vertices, then Painter colors such a pair,
yielding $s\le n+\spo(G_{n-2,q-1})=(n-q)+q+f(n-2,q-1)=f(n,q)$.  If Lister marks
at most one vertex from each part, including a part of size $1$, then Painter
colors such a part, yielding $s\le n-q+\spo(G_{n-1,q})=n-q+f(n-1,q)=f(n,q)$.
If Lister marks only single vertices from parts of size $2$, then
$s\le q+\spo(G_{n-1,q-1})=q+f(n-1,q-1)=f(n,q)$.
\end{proof}

%\caze{1}{Lister marks both vertices in some part of size $2$.}
%Painter colors a part of size $2$, leaving $G_{n-2,q-1}$.
%Lister scored at most $n$ in the first round, so the final score is at most
%$n+f(n-2,q-1)$.  Since $f(n,q)-f(n-2,q-1)=(n-q)+q=n$, the claim follows.
%
%\caze{2}{Lister marks at most one in each part, including a part of size $1$.}
%Painter colors a part of size $1$, leaving $G_{n-1,q}$.  Lister
%scored at most $n-q$ in the first round, so the final score is at most
%$n-q+f(n-1,q)$.  Since $f(n,q)-f(n-2,q)=n-q$, the claim follows.
%
%\caze{3}{Lister marks no full part.}
%Painter colors one vertex in a part of size $2$, leaving $G_{n-1,q-1}$.  Lister
%scored at most $q$ in the first round, so the final score is at most
%$q+f(n-1,q-1)$.  Since $f(n,q)-f(n-1,q-1)=q$, the claim follows.
%\end{proof}

%There is a polynomial-time algorithm for computing the matching number, so:

\begin{theorem}\label{thm:alpha2}
If $G$ is an $n$-vertex graph with $\alpha(G)\le2$, and $q$ is the maximum size
of a matching in $\Gb$, then $\spo(G)=\Sigma(G)=\CH{n-q+1}2+\CH{q+1}2$.  Thus
there is an cubic-time algorithm to determine $\spo(G)$ and $\Sigma(G)$ in
the class of graphs with independence number at most $2$.
\end{theorem}
\begin{proof}
Let $M$ be a maximum matching in $\Gb$.  Let $H$ be the supergraph of $G$
with vertex set $V(G)$ such that $M$ is the set of edges in $\Hb$.
The graph $H$ is a complete multipartite graph with $q$ parts of size $2$.
By Lemma~\ref{alpha2Sig}, Theorem~\ref{sharpbds}, Observation~\ref{monoton},
and Lemma~\ref{lem:multipart2},
\[
\CH{n-q+1}2+\CH{q+1}2=\Sigma(G)\le \spo(G)\le \spo(H)=\CH{n-q+1}2+\CH{q+1}2.
\]

By examining all triples, it can be tested whether $\alpha(G)\le2$.
There is an algorithm to find the maximum size of a matching in $\Gb$
that runs in time $O(n^{2.5})$ (\cite{MV}; see~\cite{Vaz} for a proof).
Thus $\Sigma(G)$ and $\spo(G)$ can be computed in cubic time when
$\alpha(G)\le2$.
\end{proof}

%\begin{corollary}
%If $G$ is an $n$-vertex graph with $\alpha(G) \leq 2$, then
%$\spo(G) \leq {n-q+1 \choose 2} + {q+1\choose 2}$, whre $q$ is the size of a
%largest matching in the complement of $G$.
%\end{corollary}
%\begin{proof}
%Such a graph $G$ is a subgraph of the graph $G_{t, q}$ defined in
%Lemma~\ref{lem:multipart2}, with $t = n-q$. Since $G \subset H$ implies
%$\spo(G) \leq \spo(H)$, we have
%\[ \spo(G) \leq \spo(G_{n-q, q}) = {n-q+1 \choose 2} + {q+1 \choose 2}. \]
%\end{proof}

Using Theorem~\ref{thm:alpha2} and the the formula in Theorem~\ref{thm:split1}
for $\spo(\csg{r}{s})$, which we will prove in Section~\ref{sec:treelower},
we can determine whether some graphs achieve the weaker lower bound in
Corollary~\ref{n2alph}.

\begin{theorem}
If $G$ is an $n$-vertex complete multipartite graph that is regular, then
$\spo(G)=\ntalph$ if and only if $\chi(G)=1$ or $\alpha(G)\le 2$.
\end{theorem}
\begin{proof}
Let $t=\chi(G)$ and $r=\alpha(G)$; note that $G$ is a complete $t$-partite
graph in which all parts have size $r$, and $n=rt$.  Corollary~\ref{n2alph}
yields the desired lower bound $\spo(G)\ge \FR{rt}2(1+t)$, which we write as
$r\CH{t+1}2$ (Lister guarantees this much by Observation~\ref{disjoint}, using
a covering of $V(G)$ by $r$ disjoint $t$-cliques).  When $r$ or $t$ equals $1$,
the graph is complete or empty, and we have seen that equality holds in these
cases.  When $r=2$, the graph is $G_{n,n/2}$ of Theorem~\ref{thm:alpha2},
with $t=q=n/2$, which yields $\spo(G)=2\CH{t+1}2$.

When $r>2$ and $t>1$, a strategy for Lister establishes a stronger lower bound.
In each of the first $t-1$ rounds, Lister marks $r-1$ vertices from each part
having no colored vertices.  By Observation~\ref{simple}, Painter responds
optimally by coloring all the marked vertices in one part, reducing it to one
uncolored vertex.  The remaining marked parts still have no colored vertices.
After $t-1$ rounds, Lister has scored $\SE i2t i(r-1)$, which equals
$(r-1)\CH{t+1}2-(r-1)$, and the uncolored graph is $\Kb_r\gjoin K_{t-1}$.
In Theorem~\ref{thm:split} we will compute $\spo(\Kb_r\gjoin K_{t-1})$;
the formula yields $\spo(\Kb_r\gjoin K_{t-1})\ge r+\CH t2+(t-1)(\sqrt{2r}-1)$.
Summing the contributions from the first $t-1$ rounds and the remaining graph
$\Kb_r\gjoin K_{t-1}$ yields $\spo(G)\ge r\CH{t+1}2+(t-1)(\sqrt{2r}-2)$.  This
lower bound exceeds $r\CH{t+1}2$ when $r>2$ and $t\ge2$.
\end{proof}

%When $r=2$, we prove $\spo(G)=2\CH{t+1}2$.  More generally, let $H_{t,p}$ be
%the complete $t$-partite graph with $p$ parts of size $2$ and $t-p$ parts of
%size $1$; the desired graph $G$ is $H_{t,t}$.  We prove $\spo(H_{t,p})=f(t,p)$,
%where $f(t,p)=\CH{t+1}2+\CH{p+1}2$.  Lister ensures the lower bound by playing
%separately on a $t$-clique and a $p$-clique covering $V(H_{t,p})$.
%
%For the upper bound, we use induction on $t+p$, which is the number of vertices.
%The bound is trivial for $t+p=0$.  For $t+p\ge 1$, consider the first round,
%and let $s$ be the total score.  If Lister marks some two nonadjacent vertices,
%then Painter colors such a pair, yielding
%$s\le t+p+\spo(H_{t-1,p-1})=t+p+f(t-1,p-1)=f(t,p)$.  If Lister marks at most
%one vertex from each part, including a part of size $1$, then Painter colors
%such a part, yielding $s\le t+\spo(H_{t-1,p})=t+f(t-1,p)=f(t,p)$.  If Lister
%marks only single vertices from parts of size $2$, then
%$s\le p+\spo(H_{t,p-1})=p+f(t,p-1)=f(t,p)$.

\section{All graphs such that $\spo(G)=\chisp(G)$}\label{sec:strict}
Here we prove Theorem~\ref{chisp}: Equality holds in $\spo(G)\le \chisp(G)$ if
and only if every component of $G$ is complete.  Sufficiency is immediate,
since $\spo(K_n)=\chisp(K_n)=\CH{n+1}2$.  For necessity, we show that equality
holds only when $\spo(G)=\C{V(G)}+\C{E(G)}$ and that this latter equality holds
only when every component is complete.

As noted in~\cite{CMPW}, $\C{V(G)}+\C{E(G)}$ is an easy upper bound on
$\chisp(G)$, proved by Painter playing greedily with respect to some vertex
ordering.  A graph is {\it sp-greedy}~\cite{CMPW} when equality holds in that
bound.  Our first step is to show that $\spo(G)=\chisp(G)$ only when $G$ is
sp-greedy.
%Hence we consider the $f$-painting game.
%\begin{theorem}\label{thm:spochisp}
%$\spo(G)=\chisp(G)$ if and only if $G$ is a disjoint union of complete graphs.
%\end{theorem}

In the $f$-painting game, Painter must immediately color any vertex with
no remaining tokens.  Therefore, when Lister marks a vertex $v$ having exactly
one token, Lister should also mark all neighbors of $v$.  Zhu formalized this
observation.  Let $N(v)$ denote the set of neighbors of $v$.

\begin{proposition}[Zhu~\cite{Z1}]\label{prop:simple}
If $f(v)=1$ for a token assignment $f$ on a graph $G$, then $G$ is
$f$-paintable if and only if $G-v$ is $f'$-paintable, where
\[ f'(w) =
\begin{cases}
  f(w)-1, &\textrm{if $w \in N(v)$,} \\
  f(w), &\textrm{otherwise.}
\end{cases}
\]
\end{proposition}
\begin{lemma}\label{lem:earlycolor}
Given $\spo(G)=\chisp(G)$, let $f$ be an assignment of $\chisp(G)$ tokens under
which $G$ is $f$-paintable.  Let Lister play optimally in the slow-coloring
game (guaranteeing to score at least $\spo(G)$).  If Painter interprets Lister
as playing in the $f$-painting game and responds using an optimal strategy
there, then the set colored by Painter always consists of vertices that began
the round with one token.
\end{lemma}
\begin{proof}
Under the given strategies, let $g(v)$ be the number of times that a vertex
$v$ is marked.  By Lister's strategy, $\spo(G)=\sum_{v\in G}g(v)$.  On the
other hand, $\spo(G)=\chisp(G)=\sum_{v\in G}f(v)$.  Since Painter wins,
$g(v)\le f(v)$ for all $v$, so $g(v) = f(v)$ for all $v$. Since $v$ has
$f(v)-g(v)+1$ tokens at the beginning of the round in which it is colored, the
claim follows.
\end{proof}

\begin{lemma}\label{lem:deletegreedy}
Let $G$ be a graph, and let $f$ be an assignment of $\chisp(G)$ tokens such
that $G$ is $f$-paintable.  If $G-v$ is sp-greedy for some vertex $v$ such that
$f(v)=1$, then $G$ is sp-greedy.
\end{lemma}
\begin{proof}
By Proposition~\ref{prop:simple}, the graph $G-v$ is $f'$-paintable, where 
$f'(w) = f(w) - 1$ for $w \in N(v)$ and $f'(w) = f(w)$ otherwise.  Since also
$G-v$ is sp-greedy,
\[
\sum_{w\in V(G)}f(w)-f(v)-d(v)~=
\sum_{w\in V(G-v)} f'(w)~\ge~ \C{V(G)}-1+\C{E(G)}-d(v).
\]
Thus $\sum_{w\in V(G)}f(w)\ge \C{V(G)}+\C{E(G)}+f(v)-1$.  Since $f(v)=1$,
$G$ is sp-greedy.
\end{proof}

Let $N[v]$ denote the {\it closed neighborhood} of a vertex $v$, meaning
$N[v]=N(v)\cup\{v\}$.

\begin{lemma}\label{lem:nonspgreedy}
If $G$ is not sp-greedy, then $\spo(G)<\chisp(G)$.
\end{lemma}
\begin{proof}
Let $G$ be a counterexample with fewest vertices, so $G$ is not sp-greedy, but
$\spo(G)=\chisp(G)$.  Let $f$ be an assignment of $\chisp(G)$ tokens such that
$G$ is $f$-paintable.  Let Lister play an optimal strategy in the slow-coloring
game on $G$.  Since $\spo(G)=\chisp(G)$, an optimal strategy for Painter would
be to follow an optimal strategy $\cS$ in the $f$-painting game on $G$.

Let $M$ be the set marked by Lister on the first move.  Since $\cS$ would
be optimal for Painter, by Lemma~\ref{lem:earlycolor} there exists $v \in M$
such that $f(v) = 1$.  Let $G'=G-v$, and let $M'=M-N[v]$.  Let $I'$ be
Painter's response to $M'$ in an optimal strategy $\cS'$ for the slow-coloring
game on $G'$.  The set $I'\cup\{v\}$ is independent.  Instead of using $\cS$,
Painter responds to $M$ on $G$ by coloring $I'\cup\{v\}$ and then continues
play according to $\cS'$.

By Lemma~\ref{lem:deletegreedy}, $G'$ is not sp-greedy.  Since $G$ is a minimal
counterexample, $\spo(G')<\chisp(G')$.  On the other hand,
Proposition~\ref{prop:simple} implies $\chisp(G')\le\chisp(G)-(d(v)+1)$.
Thus $d(v)+1+\spo(G')<\chisp(G)$.

In the game on $G$, the total scored by Lister is at most $1+d(v)+\spo(G')$,
since $\spo(G')$ counts everything scored in the game except $M-M'$ in the
first round.  Since $d(v)+1+\spo(G')<\chisp(G)$, this contradicts the
hypothesis that Lister can score at least $\chisp(G)$.
\end{proof}

\begin{corollary}
If $\spo(G) = \chisp(G)$, then $\spo(G) = \sizeof{V(G)} + \sizeof{E(G)}$.
\end{corollary}

The following theorem completes the proof of Theorem~\ref{chisp}.

\begin{theorem}  
$\spo(G)=\C{V(G)}+\C{E(G)}$ if and only if every component of $G$ is complete.
\end{theorem}
\begin{proof}
It suffices to show $\spo(G) < \sizeof{V(G)} + \sizeof{E(G)}$ when $G$ is
connected and not complete.  Let $M$ be the set marked by Lister on an optimal
first move in the slow-coloring game on $G$.

\caze{1}{$M = V(G)$.} Since $G$ is connected and not complete, $G$ has 
a vertex $v$ with nonadjacent neighbors $w$ and $w'$.  Let $G'=G-\{w,w'\}$,
and let $M'= V(G)-(N[w]\cup N[w'])$; note that $M'\subsetneq V(G')$.
Let $I'$ be Painter's response to $M'$ in an optimal strategy $\cS'$ on $G'$.

The set $I'\cup\{w,w'\}$ is independent.  In response to $M$ on $G$, Painter
colors $I'\cup\{w,w'\}$ and continues play according to $\cS'$.  In the game
on $G$, the total scored by Lister is at most $2+\C{N(w)\cup N(w')}+\spo(G')$,
since $\spo(G')$ counts everything scored in the game except
$\C{N[w]\cup N[w']}$ on the first round.  Since
$\spo(G')\le\C{V(G')}+\C{E(G')}$,
\[
\spo(G)\le 2+\C{N(w)\cup N(w')}+\C{V(G')}+\C{E(G')}
= \C{V(G)}+\C{E(G')}+\C{N(w) \cup N(w')}.
\]
Since $v\in N(w)\cap N(w')$, we have
$\C{E(G')}\le \C{E(G)}-(\C{N(w)\cup N(w')}+1)$.
Hence $\spo(G)\le\C{V(G)}+\C{E(G)}-1$.

\caze{2}{$\nul\ne M \subsetneq V(G)$.} Since $G$ is connected, $G$ has an edge
$vw$ with $w \in M$ and $v \notin M$.  Let $G' = G-w$ and $M' = M-N[w]$.
Let $I'$ be Painter's response to $M'$ in an optimal strategy $\cS'$ on $G'$.
In response to $M$ on $G$, Painter colors $I' \cup \{w\}$ and continues play
according to $\cS'$.  Let $M_0 = N(w) \cap M$.  Adding the part of the score
in the first round that is not counted in the game on $G'$, we have
\[
\spo(G)\le \C{M_0}+1+\spo(G')\le \C{M_0}+1+\C{V(G')}+\C{E(G')}
= \C{V(G)} + (\C{M_0} + \C{E(G')}).
\]
Since $v\notin M$, we have $\C{E(G')}\le \C{E(G)}-\C{M_0}-1$.
Hence $\spo(G)\le\C{V(G)}+\C{E(G)}-1$.
\end{proof}

\section{Bounds for $n$-Vertex Trees}\label{sec:treeupper}\label{sec:treelower}
It is easy for Lister to score $\floor{3n/2}$ on the $n$-vertex path $P_n$.
Lister first marks all $n$ vertices.  Since $\alpha(P_n)=\ceil{n/2}$, Lister
can score $\floor{n/2}$ more by marking all vertices that remain after Painter
deletes an independent set.  Indeed, $\spo(G)\ge 2n-\alpha(G)$ in any graph
by Lister marking all vertices for two rounds.
We thus can prove $\spo(T)\le\spo(P_n)$ for each $n$-vertex tree $T$ by
proving $\spo(T)\le \floor{3n/2}$.  There are several ways to prove this fact;
the efficient phrasing we present here was suggested by Xuding Zhu.

\begin{theorem}\label{treeup}
If $T$ is an $n$-vertex tree, then $\spo(T)\le\FL{3n/2}$.
\end{theorem}
\begin{proof}
We use induction on $n$; the statement holds by inspection for small $n$.
Let marking $M$ be an optimal first move for Lister.  It suffices to prove
that $M$ contains an independent set $I$ such that $\spo(T-I)\le \FR32 n-|M|$.
By Observation~\ref{simple}, we may assume that $T[M]$ is connected and that
Painter colors $X$ or $Y$, where $T[M]$ is a bipartite graph with parts $X$ and
$Y$.  It thus suffices to show that the average of $\spo(T-X)$ and $\spo(T-Y)$
is at most $\FR32 n-\C M$.

Let $o(H)$ denote the number of components of a graph $H$ having odd order.
Summing the inductive bound over all components of $T-I$ yields
$\spo(T-I)\le\FR32(n-|I|)-\FR12o(T-I)$.  We apply this computation to both
$T-X$ and $T-Y$.

Let $T'=T-E(T[M])$.  Each component of $T'$ contains exactly one vertex of $M$.
Consider a vertex $v\in X$, contained in the component $R$ of $T'$.  If $R$ is
odd, then $R$ is counted in $o(T-Y)$.  If $R$ is even, then $R$ contributes at
least $1$ to $o(T-X)$.  The symmetric statement holds for $v\in Y$.  Thus
$o(T-X)+o(T-Y)\ge |M|$.  We compute
\begin{align*}
\FR12\left(\spo(T-X)+\spo(T-Y)\right)
&\le \FR12\left(\FR32(n-|X|)-\FR12o(T-X)+\FR32(n-|Y|)-\FR12o(T-Y)\right)\\
&\le \FR12\left(3n-\FR32|M|-\FR12|M|\right)=\FR32 n-|M|
\end{align*}

\vspace{-2pc}
\end{proof}
\bigskip

Next we determine $\spo(K_{1,n-1})$ and prove $\spo(T)\ge\spo(K_{1,n-1})$ for
every $n$-vertex tree $T$.  More generally, we compute $\spo(\csg{r}{s})$.
For $k,r\in\NN\cup\{0\}$, let $t_k=\CH{k+1}2$ and $u_r=\max\{k\st t_k \le r\}$.
Note that $u_r = \floor{\frac{-1 + \sqrt{1 + 8r}}{2}}$.
The numbers of the form $\CH{k+1}2$ are the \textit{triangular numbers}.
Before computing $\spo(\csg rs)$, we need a technical lemma about $u_r$.

\begin{lemma}\label{lem:trinum}
$u_{r-u_r}=u_r$ when $r+1$ is a triangular number, and otherwise
$u_{r-u_r}=u_r-1$.
\end{lemma}
\begin{proof}
If $u_r=k$, then $t_k\le r<t_{k+1}$.  Also $t_{k+1}-t_k=k+1$.  Thus $r-k=t_k$
if $r+1=t_{k+1}$, yielding $u_{r-u_r}=u_r$.  However, $t_{k-1}\le r-k<t_k$ if
$t_k\le r\le t_{k+1}-2$, yielding $u_{r-u_r}=u_r-1$.
\end{proof}

\begin{theorem}\label{thm:split}
For $r,s\in\NN$,
\[ \spo(\csg{r}{s})=r+{s+1\choose2}+s u_r. \]
\end{theorem}
\begin{proof}
We use induction on $r+s$.  Let $f(r,s)= r + {s+1\choose2} + s u_r$.  When $r$
or $s$ is $0$, the claim clearly holds.  For $rs>0$, let $G = \csg{r}{s}$.
Also let $[r]=\{1,\dots,r\}$.  Let $R$ and $S$ denote the sets of vertices 
with degree $s$ and degree $r+s-1$, respectively.

If Lister marks no vertex of $S$, then Painter colors all marked vertices.
By the induction hypothesis, this is easily seen to be non-optimal for Lister.
Hence we may assume that Lister marks some vertex of $S$.
Since Painter can color at most one vertex of $S$ in response, Lister should
mark all of $S$ plus perhaps some of $R$.  Painter responds by coloring one
vertex of $S$ or all marked vertices of $R$.  Applying the recurrence of
Proposition~\ref{pr:recur} and the induction hypothesis, 
\[
\spo(\csg{r}{s}) =
\max_{k\in [r]}((k+s)+\min\{\spo(\csg{r-k}{s}),\spo(\csg{r}{s-1})\}
=\max_{k \in [r]}g(k),
\]
where
\[ g(k)=k+s+\min\{f(r-k,s),f(r,s-1)\}. \]

By the induction hypothesis, $g(k)$ is the best result Painter can obtain when
Lister marks $S$ and $k$ vertices of $R$ on the first round.  We compute
\begin{align*}
g(u_r)&= u_r+s+\min\{r-u_r+\CH{s+1}2+su_{r-u_r},r+\CH s2+(s-1)u_r\}\\
&=\min\{r+\CH{s+1}2+s(u_{r-u_r}+1),r+\CH{s+1}2+su_r\}\\
&=\min\{f(r,s)+s(1+u_{r-u_r}-u_r),f(r,s)\}.
\end{align*}
By Lemma~\ref{lem:trinum}, $u_r-u_{r-u_r}\in\{0,1\}$, so $g(u_r)=f(r,s)$.
Furthermore, if Lister marks $u_r$ vertices in $R$ and all of $S$, then
deleting a vertex of $S$ is an optimal response for Painter. 

We seek $\max_k g(k)$.  Note that $g(0)=s+f(r,s-1)$, since $f(r,s)>f(r,s-1)$.
While $f(r-k,s)\ge f(r,s-1)$, we have $g(k)=g(0)+k$, so in this range we
maximize $k$.  Since also $f(r-k,s)$ decreases as $k$ increases, and
$f(r-k,s)+k$ is nonincreasing, subsequently $g$ is nonincreasing.  Hence $g(k)$
is maximized by the largest $k$ such that $f(r-k,s)\ge f(r,s-1)$.  To show this
is $u_r$, it suffices to show $f(r-u_r, s) \geq f(r,s-1)$ and
$f(r-(u_r+1), s) < f(r,s-1)$.

When $r+1$ is not a triangular number, Lemma~\ref{lem:trinum} yields
$f(r-u_r, s) = f(r, s-1)$.  Since $f(r-(u_r+1),s)<f(r-u_r,s)$, the desired
value of $k$ is $u_r$.

When $r+1$ is a triangular number, Lemma~\ref{lem:trinum} yields
$f(r-u_r, s) = f(r, s-1) + s$.  Since $r$ itself then is not a triangular
number, Lemma~\ref{lem:trinum} and $u_{r-1} = u_r$ yield
$f(r-(u_r+1),s)=f(r-1,s-1)<f(r,s-1)$.  Again the desired value is $u_r$.
\end{proof}

Setting $s=1$, we have $\spo(K_{1,n-1})=n+u_{n-1}=n+\FL{\FR{-1+\sqrt{8n-7}}2}$.

\begin{theorem}
If $T$ is an $n$-vertex tree, then $\spo(T) \geq \spo(K_{1,n-1}) = n + v_n$,
where $v_n=u_{n-1}$.
\end{theorem}
\begin{proof}
We use induction on $n$.  Since $\spo(P_n)=\FL{3n/2}$ and always
$\FL{n/2}\ge v_n$, the claim holds when $T$ is a path.  Hence also the claim
holds for $n\le4$.

The main idea is that Lister can play separately on disjoint induced subgraphs,
yielding $\spo(T)\ge \spo(T_1)+\spo(T_2)$ when $T_1$ and $T_2$ are the
components obtained by deleting an edge of $T$.  If $n_i=\C{V(T_i)}$, then
$\spo(T)\ge n+v_{n_1}+v_{n_2}$.  It therefore suffices to find an edge $e$ such
that $v_{n_1} + v_{n_2} \geq v_{n}$.  We may assume $n_1\le n_2$.

When $n\ge5$, we have $v_n\le 1+v_{n-3}$.  If $T$ has an edge whose deletion
leaves a component with two or three vertices, then $v_{n_1}=1$, and
$v_{n_1}+v_{n_2}\ge 1+v_{n-3}\ge v_n$, as desired.  If $T$ is not a star,
then $T$ has an edge not incident to a leaf, and when $n\le7$ every edge not
incident to a leaf has this property that $n_1\le3$.

In the remaining case, $n\ge8$ and $n_1,n_2\ge4$.  When $n\ge8$, we have
$v_n\le 1+v_{n-4}$.  Let $g(x)=\FR{-1+\sqrt{8n-7}}2$; note that $v_n=\FL{g(n)}$.
We need $v_{n_1}+v_{n_2}\ge v_n$.

Let $p=4$ and $q=n-4$.  Since $g$ is concave, $g(p+x)-g(p)\ge g(q)-g(q-x)$
when $0\le x\le n/2-4$, which yields $g(p+x)+g(q-x)\ge g(p)+g(q)$.  If
$a+b\ge c+d$, then $\FL{a}+\FL{b}\ge \FL{a+b}-1\ge\FL{c+d}-1\ge\FL{c}+\FL{d}-1$.
Applying this with $(a,b,c,d)=(g(n_1),g(n_2),g(p),g(q))$, and using $v_4=2$,
we obtain $v_{n_1}+v_{n_2}\ge v_4+v_{n-4}-1=1+v_{n-4}\ge v_n,$
as desired.
\end{proof}

\section{Bounds for Complete Bipartite Graphs}\label{sec:bip}
We prove upper and lower bounds on $\spo(K_{r,s})$ separately, via strategies
for Painter and Lister.  Together, these results yield Theorem~\ref{thm:bip}.
Our general upper bound is fairly good when $r$ is much larger than $s$, but
for $K_{r,r}$ it still differs from the lower bound in the leading coefficient.

%After a general lower bound, we prove a better lower bound for $\spo(K_{r,r})$;
%it is not clear how to generalize the better bound to unbalanced complete
%bipartite graphs.  Subsequently, we prove an upper bound.
%
%\begin{proposition}\label{thm:starbound}
%For $r\in\NN$, let $P_s(r)$ denote the set of partitions of the integer $r$
%with exactly $s$ parts.
%For $r\ge s$,
%\[
%\spo(K_{r,s})\ge
%r+s+\max\left\{\SE i1s u_{\lambda_i}\st\lambda\in P_s(r)\right\}
%\]
%In particular, $\spo(K_{r,s}) \geq r+s+su_{\floor{r/s}}$.
%\end{proposition}
%\begin{proof}
%Lister plays on the disjoint subgraphs
%$K_{1,\lambda_1},\ldots,K_{1,\lambda_s}$ (the center of each star lies in the
%part of size $s$).  Since $\spo(K_{1,t}) = t+1+u_t$, Observation~\ref{disjoint}
%yields the lower bound.  Since $u_t$ is concave in $t$, the partition of $r$
%into $s$ nearly equal parts is best.
%\end{proof}
%
%Since $u_t \approx \sqrt{2t}$ for large $t$, the lower bound for highly
%unbalanced graphs is roughly $r+s+\sqrt{2rs}$.  When $r=s$, the stars each
%yield score $3$, and the lower bound simplifies to $3r$.  Lister achieves
%$3r$ by playing disjoint edges or by marking the entire set twice.
%A slight modification yields a slightly better lower bound.
%
%\begin{proposition}
%  $\spo(K_{r,r}) \geq 3r-1+u_r \approx 3r + \sqrt{2r}$.
%\end{proposition}
%\begin{proof}
%Lister starts by marking $r-1$ vertices in each part, scoring $2r-2$.  Painter
%may only delete vertices from one part, so Painter's response leaves $K_{1,r}$.
%Optimal play by Lister on $K_{1,r}$ scores $r+1+u_r$ more.
%\end{proof}

\begin{remark}\label{remark}
{\it Properties of optimal strategies.}
In light of symmetry, we can describe a move by Lister in the slow-coloring
game on $K_{r,s}$ as a pair $(j,i)$, marking $j$ vertices in the part $X$
of size $r$ and $i$ in the part $Y$ of size $s$.  An optimal response by
Painter colors all marked vertices in one part, and the cost under optimal play
will be $j+i+\min\{\spo(K_{r-j,s}),\spo(K_{r,s-i})\}$.

Since Lister can restrict play to an induced subgraph, $\spo(K_{r,s})$
increases with $r$ and with $s$ (in fact strictly, by
Observation~\ref{disjoint}).  If an optimal response by Painter for the move $(j,i)$ is to color in
$X$, and $j'>j$, then
$$
\spo(K_{r-j',s})<\spo(K_{r-j,s})\le \spo(K_{r,s-i}),
$$
and hence it is also optimal to color in $X$ in response to $(j',i)$.
Therefore, given $i$, there is a threshold $\hj$ such that in response
to $(j,i)$, Painter should color the $j$ vertices in $X$ when $j\ge \hj$ and
should color the $i$ vertices in $Y$ when $j<\hj$.  As a result, we specify a
Painter strategy by specifying $\hj$ as a function of $i$ when playing on
$K_{r,s}$.
%Furthermore, the same argument with the parts interchanged shows that $\hj$
%should not decrease when $i$ increases.
\end{remark}

\begin{theorem}\label{biplower}
$\spo(K_{r,s}) \leq r+s+2\sqrt{rs}$.
\end{theorem}
\begin{proof}
Let $f(r,s)=r+s+2\sqrt{rs}$.  We prove $\spo(K_{r,s}) \leq f(r,s)$ by induction
on $r+s$, with basis $s=0$, where $\spo(K_{r,s}) = r = f(r,0)$.  For $r+s>0$,
without loss of generality we may assume $r\ge s$.  Let $X$ and $Y$ be the
parts of the bipartition, with $\sizeof{X} = r$ and $\sizeof{Y} = s$.

As explained in Remark~\ref{remark}, we specify $\hj$ as a function of $i$
so that in response to Lister's first move $(j,i)$, Painter colors the $j$
marked vertices in the part of size $r$ if $j\ge \hj$ and otherwise colors
the $i$ marked vertices in the part of size $s$.  When $i=0$, the threshold
is $0$.  Painter colors all the marked vertices, and the induction hypothesis
yields the desired bound.  Hence we may assume $i>0$.

Below the threshold, the remaining game is on $K_{r,s-i}$, independent of $j$.
Hence in this case Lister should maximize $j$ to gain the most initial cost,
and $i+\hj+\spo(K_{r,s-i})$ is an upper bound on the total cost.  Above the
threshold, by the induction hypothesis the value $i+j+f(r-j,s)$ is an upper
bound on the total cost, given that Painter's strategy colors the $i$ vertices
in $Y$.  The value is $i+r+s+2\sqrt{(r-j)s}$.  As a continuous function of
$j$, this value is strictly decreasing.  We consider only values at least
$\hj$, so $i+\hj+f(r-\hj,s)$ is an upper bound on the total cost when Lister's
play against Painter's strategy is in this range; this statement does not
require $\hj$ to be an integer.

Using the induction hypothesis, we compute
\[
\begin{array}{r@{}l@{}l}
\spo(K_{r,s}) &\le\max_i\big[i+\hj+\max\{f(r,s-i),f(r-\hj,s)\}\big]\\
&=r+s+\max_i\big[\max\{\hj+2\sqrt{r(s-i)}, i+2\sqrt{(r-\hj)s}\}\big].
\end{array}
\]
It then suffices to prove the following inequalities for $1\le i\le s$ under
an appropriate threshold function $\hj$:
$$
\hj+2\sqrt{r(s-i)}\le i+2\sqrt{(r-\hj)s} \le 2\sqrt{rs}.
$$

Define the threshold function for Painter by $\hj={{i\sqrt{r/s}}}$.
% except $J=1$ when $i=1$ and $r=s$.
The inequality on the right is equivalent to
$2\ge \FR{i}{\sqrt{rs}-\sqrt{rs-sJ}}=\FR{i(\sqrt{rs}+\sqrt{rs-sJ})}{sJ}$.  Thus
it suffices to have $\FR{2i\sqrt{rs}}{s\hj}\le2$, as is given.

The inequality on the left is $\hj-i\le2(\sqrt{rs-sJ}-\sqrt{rs-ir})
=2\frac{ir-s\hj}{\sqrt{rs-s\hj} + \sqrt{rs-ir}}$. 
Thus it suffices to prove $\hj-i\le\frac{ir-s\hj}{\sqrt{rs}}$, which simplifies
to $\hj\le\FR{ ir+i\sqrt{rs}}{ s+\sqrt{rs}}$.  Since $J={i\sqrt{r/s}}$, it
suffices to prove $\FR{i\sqrt{rs}}{ s} \le \FR{ ir+i\sqrt{rs}}{ s+\sqrt{rs}}$.
Cross-multiplying shows that the two sides are equal.
\end{proof}

When $r=s$, Theorem~\ref{biplower} yields $\spo(K_{r,r}) \le 4r$.  Equality
never holds in this bound (it equals the upper bound $|V(G)|\rho(G)$ from
Theorem~\ref{sharpbds}, which only holds with equality when $G$ has no edges)
but we believe it is asymptotically sharp.  Since every induced
subgraph of a complete bipartite graph is a complete bipartite graph, storing
the optimal values for smaller graphs makes it relatively easy to explore all
options for the moves in the first round to compute $\spo(K_{r,s})$.  The
resulting data for $\spo(K_{r,r})$ with $r\le 1500$ is very closely explained
(always with error at most $2$) by $4r-\sqrt r-\log_3 r$.  We conjecture that
$4r-o(r)$ is correct.

Now consider the lower bound.  When $r=s$, a lower bound of $3r$ follows from
Lister playing the game separately on each edge of a matching.  Since
$\spo(K_{3,3})=10$ (again computed from smaller values by exploring all options
in the first round), the coefficient can be improved to $10/3$.  Indeed, since
data suggests that the ratio tends to $4$, successively larger ratios from
bigger examples yield better asymptotic lower bounds.  The computational data
thus allows us to give lower bounds on $\spo(K_{r,r})/r$ that seem to approach
$4$.  However, such bounds would at present only be proved by a long chain of
case analysis.  Instead, we give a relatively short proof a general lower
bound for all $K_{r,s}$ that reduces in the case $r=s$ to $(7r-3)/2$.

Recall that $\spo(K_{1,t})=t+1+u_{t}$, where $u_t=\floor{(-1+\sqrt{8t+1})/2}$
(Theorem~\ref{thm:split}).  Our strategy for Lister when $r>s+1$ is based on
the results for stars.

\begin{theorem}
Let $u_t=\FL{\FR{-1+\sqrt{8t+1}}2}$.
If $r \geq s \ge1$, then $\spo(K_{r,s}) \geq r+\FR{5s-3}2+u_{r-s}$.
%f(r,s)$, where 
%\begin{equation}\label{eq:fdef}
%f(r,s) =
%\begin{cases}
%  r + \frac{5s-3}{2} + u_{r-s} &\text{if $r > s>0$,} \\  
%  r + \frac{5s-2}{2}  &\text{if $r = s>0$,}\\
%  r&\text{if $s=0$.}
%\end{cases}
%\end{equation}
\end{theorem}
\begin{proof}
Let $f(r,s)=r+\FR{5s-3}2+u_{r-s}$.  Given $\spo(K_{r,1})$ from 
Theorem~\ref{thm:split}, the value $f(r,1)$ is a lower bound when $s=1$ because
$u_r\ge u_{r-1}$.  Note that $u_0=0$.  By a short case analysis considering
all possible moves, $\spo(K_{2,2})=6>2+\FR{5\cdot2-3}2=f(2,2)$ and
$\spo(K_{3,2})=8>3+\FR{5\cdot2-3}2=f(3,2)$.  Hence in a proof by induction on
$r+s$ we may assume $r\ge s\ge 2$ with $(r,s)\ne(3,2)$.

%When $s=1$, this follows from the earlier theorem about $\spo(K_{1,r})$, as
%$c \geq 3/2$ implies that $f(r,1) \geq r+1+u_r$. Similarly, one can check by
%brute force that $c \geq 3/2$ guarantees $\spo(K_{r,s}) \geq f(r,s)$
%for $s \leq r \leq 3$.

\smallskip
{\it Lister strategy:}\\
When $r \geq s+2$, Lister marks one vertex in the small part and $u_{r-s}$ in
the large part.\\
When $r = s+1$, Lister marks two vertices from each part.\\
When $r = s$, Lister marks one vertex from each part.
\smallskip

In each case, we consider both possible responses by Painter.  We show that
Lister achieves cost at least $f(r,s)$ for each response.

\caze{1}{$r \geq s+2$.}
By the induction hypothesis,
$$
\spo(K_{r,s}) \geq (1+u_{r-s}) + \min\{f(r-u_{r-s}, s), f(r,s-1)\}.    
$$

Note that $r-u_{r-s}>s$ when $r-s\ge2$, since $u_t<t$ when $t\ge2$.
Thus $f(r-u_{r-s},s)=r-u_{r-s}+\FR{5s-3}2+u_{r-s-u_{r-s}}$.
By Lemma~\ref{lem:trinum}, $u_t\le 1+u_{t-u_t}$.  Thus
$$
1+u_{r-s}+f(r-u_{r-s},s)=r+\FR{5s-3}2+1+u_{r-s-u_{r-s}}\ge f(r,s).
$$
Also,
\begin{align*}
1+u_{r-s}+f(r,s-1)&=u_{r-s}+r+\FR{5(s-1)-3}2+1+u_{r-(s-1)}\\
&\ge r+\FR{5s-3}2+u_{r-s}-\FR32+u_3>f(r,s).
\end{align*}

\caze{2}{$r = s+1\ge4$.}
Here $K_{r-2,s}=K_{r-1,s-1}$, so the induction hypothesis yields
$$
\spo(K_{r,s}) \geq 4 + \min(f(r-1,s-1), f(r,s-2)).
$$
We compute
$$4+f(r-1,s-1)=r-1+\FR{5(s-1)-3}2+4+u_1=r+\FR{5s-3}2+\FR12+u_{r-s}>f(r,s).$$
Since $s\ge3$, we have $s-2\ge1$, and hence
$$4+f(r,s-2)=r+\FR{5(s-2)-3}2+4+u_3=r+\FR{5s-3}2-1+1+u_1=f(r,s).$$

\caze{3}{$r=s$.}
Painter's response always leaves $K_{r,r-1}$.  By the induction hypothesis,
$$
\spo(K_{r,s})\ge 2+f(r,r-1)= r+\FR{5(r-1)-3}2+2+u_1=r+\FR{5r-2}2>f(r,r).
$$

\vspace{-2pc}
\end{proof}

\providecommand{\bysame}{\leavevmode\hbox to3em{\hrulefill}\thinspace}
\providecommand{\MR}{\relax\ifhmode\unskip\space\fi MR }
\providecommand{\MRhref}[2]{%
  \href{http://www.ams.org/mathscinet-getitem?mr=#1}{#2}
}
\providecommand{\href}[2]{#2}

\end{document}